\documentclass[a4paper]{amsart} 
\usepackage{amssymb} 
\usepackage{tikz-cd}
\usepackage{hyperref}

\title{Linear versus set valued  Kronecker representations}
\author{Kay Gro{\ss}blotekamp}
\author{Henning Krause}

\theoremstyle{plain}
\newtheorem{thm}{Theorem}[section]
\newtheorem{prop}[thm]{Proposition}
\newtheorem{lem}[thm]{Lemma} 

\newtheorem{cor}[thm]{Corollary}

\theoremstyle{definition}

\newtheorem{exm}[thm]{Example}

\theoremstyle{remark}

\numberwithin{equation}{section}

\newcommand{\card}{\operatorname{card}}

\newcommand{\Hom}{\operatorname{Hom}}

\newcommand{\Ring}{\operatorname{\mathsf{Ring}}}
\newcommand{\set}{\operatorname{\mathsf{set}}}
\newcommand{\Set}{\operatorname{\mathsf{Set}}}
\renewcommand{\Vec}{\operatorname{\mathsf{Vec}}}
\renewcommand{\vec}{\operatorname{\mathsf{vec}}}

\newcommand{\op}{\mathrm{op}}

\newcommand{\lto}{\longrightarrow}
\newcommand{\smatrix}[1]{\left[\begin{smallmatrix}#1\end{smallmatrix}\right]}

\newcommand{\xto}{\xrightarrow}

\def\A{\mathsf A} 
 
\def\C{\mathsf C}
\def\D{\mathsf D}

\def\K{\mathsf K}

\def\bbF{\mathbb F}

\def\a{\alpha}

\def\Ga{\Gamma}

\begin{document}

\begin{abstract}
  A set valued representation of the Kronecker quiver is nothing but a
  quiver.  We apply the forgetful functor from vector spaces to sets
  and compare linear with set valued representations of the Kronecker
  quiver.
\end{abstract}


\date{June 5, 2016}

\maketitle

\section{Introduction}

We consider the Kronecker quiver
\[\begin{tikzcd}
\K_2&\circ\arrow[yshift=0.75ex]{r} \arrow[yshift=-0.75ex]{r}&\circ
\end{tikzcd}\]
and study its representations. Given any category $\C$, a
Kronecker representation  in $\C$ is by definition a pair of parallel
morphisms 
\[\begin{tikzcd}
X_1\arrow[yshift=0.75ex]{r}{f} \arrow[yshift=-0.75ex,swap]{r}{g}& X_0
\end{tikzcd}\]
in $\C$. We may view $\K_2$ as a category with two objects, two
identical morphisms, and two parallel morphisms that connect both
objects. Then a Kronecker representation is nothing but a functor
$\K_2\to\C$ and the morphisms between representations are by
definition the natural transformations. We denote the category of
Kronecker representations by $\C^{\K_2}$. There are two natural
examples.

\subsection*{Set valued Kronecker representations}
Take for $\C$ the category $\Set$ of sets. Then a Kronecker
representation is just a quiver, because a quiver is a quadruple
$\Ga=(\Ga_0,\Ga_1,s,t)$ consisting of a set of vertices $\Ga_0$, a set
of arrows $\Ga_1$ and two maps $s,t\colon \Ga_1\to\Ga_0$ that assign
to each arrow its start and its terminus \cite{Ga1972}. Thus
$\Set^{\K_2}$ identifies with the category of quivers.

\subsection*{Linear Kronecker representations} 
Fix a field $k$ and consider for $\C$ the category
$\Vec_k$ of vector spaces over $k$. Then the Kronecker representations
are pairs of $k$-linear maps, and these have been studied by Leopold
Kronecker \cite{Kr1890}.

\subsection*{The forgetful functor and its adjoint}

Any functor $F\colon\C\to\D$ induces a functor
$F^{\K_2}\colon\C^{\K_2}\to\D^{\K_2}$ by composing representations in
$\C$ with $F$.

There is an adjoint pair of functors
\[|-|\colon\Vec_k\lto\Set\qquad\text{and}\qquad k^{(-)}\colon \Set\lto
\Vec_k.\]
The first one is the forgetful functor and the second one is its left
adjoint which takes a set $X$ to the vector space $ k^{(X)}$ of maps
$f\colon X\to k$ such $f(x)=0$ for all but a finite number of elements
$x\in X$. Composing representations with these functors yields another
pair of adjoint functors
\[|-|^{\K_2}\colon\Vec_k^{\K_2}\lto\Set^{\K_2}\qquad\text{and}\qquad
{k^{(-)}}^{\K_2}\colon \Set^{\K_2}\lto \Vec_k^{\K_2}\]
but we simplify (and abuse) notation by writing $|-|$ instead of
$|-|^{\K_2}$ and similarly $k^{(-)}$ instead of ${k^{(-)}}^{\K_2}$.

The subject of this work is the study of the forgetful functor and its
left adjoint between the categories of linear and set valued Kronecker
representations. We restrict ourselves to representations in the
category of finite dimensional vector spaces and the category of
finite sets.

\section{Linear representations}

Fix a field $k$ and let $\vec_k$ denote the category of finite
dimensional vector spaces over $k$. For a vector space $V$ set
$V^*:=\Hom_k(V,k)$.

We consider the category $\vec_k^{\K_2}$ of $k$-linear Kronecker
representations. From the Krull-Remak-Schmidt theorem it follows that
each representation decomposes essentially uniquely into a finite
direct sum of indecomposable representations. 

It is well known that each indecomposable representation is either
\emph{preprojective}, \emph{preinjective}, or \emph{regular}.
Following an idea of A.~Hubery, we describe these as follows.  

For each integer $n\ge 0$ let $V_n$ denote the $n+1$-dimensional space
of homogeneous polynomials of degree $n$ in two variables $x$ and $y$
of degree $1$.  It is convenient to set
$V_{-1}=0$. Thus \[k[x,y]=\bigoplus_{n\ge 0}V_n.\]

The indecomposable preprojective representations are of the form
\begin{align*}
P(n)\quad& \begin{tikzcd}[ampersand replacement=\&] V_{n-1}\arrow[yshift=0.75ex]{r}{x} \arrow[yshift=-0.75ex,
swap]{r}{y}\& V_{n}
\end{tikzcd}& (n&\ge 0)
\intertext{and the  indecomposable preinjective representations are the dual
representations}
I(n)\quad &\begin{tikzcd}[ampersand replacement=\&]  V_{n}^*\arrow[yshift=0.75ex]{r}{x^*}
\arrow[yshift=-0.75ex, swap]{r}{y^*}\& V_{n-1}^* 
\end{tikzcd}& (n&\ge 0).  \intertext{Each $0\neq f\in V_{n}$ gives
  rise to a regular representation}
R(f)\quad&\begin{tikzcd}[ampersand replacement=\&]
    V_{n-1}\arrow[yshift=0.75ex]{r}{x} \arrow[yshift=-0.75ex,
    swap]{r}{y}\& V_{n}/\langle f\rangle
\end{tikzcd}
\end{align*}
where $\langle f\rangle$ is the $k$-linear subspace generated by
$f$. Note that $R(fg)\cong R(f)\oplus R(g)$ when $f$ and $g$ are
coprime.  If $f=p^d$ for some irreducible polynomial $p$ and an
integer $d\ge 1$, then $R(f)$ is indecomposable, and all
indecomposable regular representations are of this form.

\section{The forgetful functor}

Let $\set$ denote the category of finite sets. We fix a finite field
$k$ and consider the forgetful functor
$|-|\colon\vec_k^{\K_2}\to\set^{\K_2}$. It takes a linear
representation
\[\begin{tikzcd}
X_1\arrow[yshift=0.75ex]{r}{f} \arrow[yshift=-0.75ex,swap]{r}{g}& X_0
\end{tikzcd}\]
to the set valued representation
\[\begin{tikzcd}
{| X_1 |} \arrow[yshift=0.75ex]{r}{|f|} \arrow[yshift=-0.75ex,swap]{r}{|g|}& {|X_0|}.
\end{tikzcd}\]
This functor preserves products since it is a right adjoint. Thus it
suffices to describe its action on the indecomposable linear
representations.

Let us introduce two types of quivers that are relevant, namely
\emph{linear} and  \emph{cyclic} quivers:
\[\begin{tikzcd}
\mathsf A_r&1 \arrow[r]&2 \arrow[r]&\cdots \arrow[r]&r-1 \arrow[r]&r \\
\mathsf C_r&1 \arrow[r]&2 \arrow[r]&\cdots \arrow[r]&r-1 \arrow[r]&r
\arrow[bend left=15]{llll}
\end{tikzcd}\]

\subsection*{Preprojective representations} 

Fix an integer $n\ge 0$ and $f\in V_n$. We set 
\[d_x(f):= \max\{r\ge 0\mid x^r\text{ divides }f\}\quad\text{and}\quad
d_y(f):= \max\{r\ge 0\mid y^r\text{ divides }f\}.
\]
If $f=0$ or $d_y(f)=0$, then we define a quiver $\Ga(f)$ as follows. For
$f=0$ set
\[\begin{tikzcd}
\Ga(f) &0 \ar[loop right]{r}{0}
\end{tikzcd}\]
and for $d_y(f)=0$ set $d:=d_x(f)$ and
\[\begin{tikzcd}
  \Ga(f)&f \arrow[rr, "x^{-1}f"]&& x^{-1}fy \arrow[rr, "x^{-2}fy"]&&
  \cdots \arrow[rr, "x^{-d}fy^{d-1}"] &&x^{-d}fy^{d}.
\end{tikzcd}\] 
Note that $\Ga(f)$ is  a linear quiver with $d_x(f)+1$ vertices when $f\neq 0$.

\begin{prop}
  The quiver $|P(n)|$ equals the disjoint union of the quivers $\Ga(f)$
  where $f\in V_n$ such that $f=0$ or $d_y(f)=0$.
\end{prop}
\begin{proof}
  At each vertex there is at most one arrow starting and at most one
  arrow ending, since multiplication with $x$ and $y$ provides
  injective maps $V_{n-1}\to V_{n}$. Thus each connected component of
  $|P(n)|$ is either linear or cyclic, and it is easily checked that
  each component is identified by the vertex $f\in V_n$ satisfying
  $f=0$ or $d_y(f)=0$.
\end{proof}

Let $(A,a)$ be a pointed set and $n\ge 0$. We define a quiver
$\Ga((A,a),n)$ with set of vertices $A^{n+1}$ and set of arrows
$A^{n}$ such that the start and terminus of an arrow
$\mathbf a=(a_1,a_2,\ldots,a_n)$ are given as follows:
\[\begin{tikzcd}
(a,a_1,\ldots,a_{n})\arrow[r,"\mathbf a"]&(a_1,\ldots,a_n,a)
\end{tikzcd}\]

\begin{cor}
View the field $k$ as  pointed set $(k,0)$. Then the quiver
$|P(n)|$ identifies with $\Ga((k,0),n)$.\qed
\end{cor}

\begin{exm}
Let $k=\bbF_2$. Then $|P(2)|$ has $3$ connected components.
\[
\begin{tikzcd}
(0,0,0) \ar[loop right]{r}{(0,0)}&&
{(0,1,1)}\arrow[r,"{(1,1)}"]&{(1,1,0)}\\
{(0,0,1)}\arrow[r,"{(0,1)}"]&{(0,1,0)}\arrow[r,"{(1,0)}"]&{(1,0,0)}
\end{tikzcd}
\]
\end{exm}

\subsection*{Preinjective representations} 

Fix an integer $n \ge 0$. Following \cite{dB1946}, the \emph{de Bruijn
  graph} of dimension $n$ on a set $A$ is the quiver with set of
vertices $A^n$ and set of arrows $A^{n+1}$ such that the start and
terminus of an arrow $\mathbf a=(a_0,a_1,\ldots,a_n)$ are given as
follows:
\[\begin{tikzcd}
(a_0,\ldots,a_{n-1})\arrow[r,"\mathbf a"]&(a_1,\ldots,a_n)
\end{tikzcd}\]

\begin{prop}
  The quiver $|I(n)|$ identifies with the de Bruijn graph of dimension
  $n$ on the set $k$.
\end{prop}
\begin{proof}
The identification goes as follows. We use dual basis elements
$\xi^i\upsilon^{j}$ in $V_{n}^*$ corresponding to
$x^i y^{j}$ in $V_{n}$, and we identify 
\[k^{n+1}\stackrel{\sim}\lto V_{n}^*,\qquad
(a_0,a_1,\ldots,a_n)\mapsto\sum_{i+j=n} a_j\xi^i\upsilon^{j}.\]
The map $x^*\colon V_{n}^*\to V_{n-1}^*$ sends
$\sum_{i+j=n} a_j\xi^i\upsilon^{j}$ to
$\sum_{i+j=n} a_j\xi^{i-1}\upsilon^{j}$, while $y^*$ sends
$\sum_{i+j=n} a_j\xi^i\upsilon^{j}$ to
$\sum_{i+j=n} a_j\xi^{i}\upsilon^{j-1}$, where $\xi^{-1}=0=\upsilon^{-1}$.
\end{proof}

\begin{exm}
Let $k=\bbF_2$. Then $|I(2)|$ is the de Bruijn graph  of dimension $2$
on a set with $2$ elements.
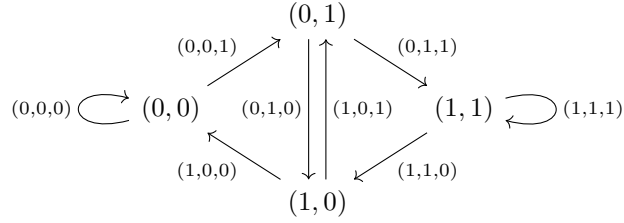
\[\begin{tikzcd}
{}&(0,1)\arrow{rd}{(0,1,1)}\arrow[xshift=-0.75ex,swap]{dd}{(0,1,0)}\\
(0,0) \ar[loop left]{l}{(0,0,0)}\arrow{ru}{(0,0,1)}&&
(1,1) \ar[loop right]{r}{(1,1,1)}\arrow{ld}{(1,1,0)}\\
{}&(1,0) \arrow{lu}{(1,0,0)}\arrow[xshift=0.75ex,swap]{uu}{(1,0,1)}
\end{tikzcd}\]
\end{exm}

\subsection*{Regular representations} 

We fix $0\neq f\in V_{n}$ and consider the corresponding regular representation
\[\begin{tikzcd}
R(f)&V_{n-1}\arrow[yshift=0.75ex]{r}{x} \arrow[yshift=-0.75ex,
swap]{r}{y}& V_{n}/\langle f\rangle 
\end{tikzcd}\]
Suppose that $R(f)$ is indecomposable. Then one of the maps given by
multiplication with $x$ and $y$ is bijective. First we consider the
case that both maps are bijective.

\begin{lem}\label{le:cyclic}
Let $\Ga=(\Ga_0,\Ga_1,s,t)$ be a finite quiver. Then the maps $s$ and $t$
are bijective if and only if $\Ga$ is a disjoint union of cyclic quivers.
\end{lem}
\begin{proof}
The  maps $s$ and $t$
are bijective if and only if at each vertex there is precisely one
arrow starting and one arrow ending.
\end{proof}

\begin{prop}
Let $R(f)$ be a regular representation and suppose that the maps given by
multiplication with $x$ and $y$ are bijective. Then the quiver
$|R(f)|$ is a disjoint union of cyclic quivers.
\end{prop}
\begin{proof}
Apply Lemma~\ref{le:cyclic}.
\end{proof}

\begin{exm}
Let $a\in k^\times$ and consider the representation
\[\begin{tikzcd}
k\arrow[yshift=0.75ex]{r}{1} \arrow[yshift=-0.75ex,swap]{r}{a}& k
\end{tikzcd}\]
which is isomorphic to $R(ax-y)$. For $k=\bbF_5$ and $a=4$ the
corresponding quiver is the following:
\[
\begin{tikzcd}
0\ar[loop right]{r}{0}&& 1 \arrow[yshift=0.75ex]{r}{1}&4
\arrow[yshift=-0.75ex]{l}{4}
&& 2 \arrow[yshift=0.75ex]{r}{2}&3 \arrow[yshift=-0.75ex]{l}{3}
\end{tikzcd}
\]
For $k=\bbF_5$ and $a=2$ one obtains the following:
\[
\begin{tikzcd}
0\ar[loop right]{r}{0}&& 1 \arrow{r}{1}&2 \arrow{r}{2}&4 \arrow{r}{4}&3
\arrow[bend left=15]{lll}{3}
\end{tikzcd}
\]
\end{exm}

It remains to consider the case that one of the maps given by
multiplication with $x$ and $y$ is not bijective. We may assume that
this is $y$. Note that the kernel is isomorphic to $k$. The
endomorphism of $V_{n}/\langle f\rangle$ given
by multiplication with $x^{-1}y$ is nilpotent, say of index $d$. 

We consider the quiver $|R(f)|$ and set $q:=\card (k)$. Observe that at
each vertex there is precisely one arrow starting and there are either
$q$ arrows ending or none.  Given a vertex $v$ that is not the
end of an arrow, there is a unique path of length $d$ that starts at
$v$ and ends at $0$.

A quiver is called a \emph{complete directed tree of height $d$ and  width $q$}, if
\begin{enumerate}
\item there is a unique vertex (called \emph{root}) where no arrow starts;
\item at each other vertex  a unique arrow starts;
\item at each vertex  either $q$ or no arrows end; 
\item each vertex where no arrow ends is connected with the root by a
  unique path of length $d$.
\end{enumerate}

\begin{prop}\label{pr:tree}
Let $R(f)$ be an indecomposable regular representation such that
  $f=p^d$ for some irreducible polynomial $p$.  Suppose that the map
  given by multiplication with $y$ is not bijective. Then the quiver
  $|R(f)|$ has a unique loop at $0$, and after removing this loop it is
  a directed tree that is complete of height $d$ and width $\card(k)$,
  except that there are only $\card(k)-1$ arrows ending at the root $0$.\qed
\end{prop}

If $R(f)$ is an indecomposable regular representation such that
multiplication with $x$ is not bijective, then one obtains  $|R(f)|$  from
$|R(f')|$ by reversing all arrows, where $f'$ is obtained from $f$ by
interchanging $x$ and $y$.

A quiver that is isomorphic or anti-isomorphic to one arising in
Proposition~\ref{pr:tree} is called \emph{almost complete directed
  tree of width $q$}.

\begin{exm}
Consider the representation
\[\begin{tikzcd}[ampersand replacement=\&]
k^2\arrow[yshift=0.75ex]{rr}{\smatrix{1&0\\ 0&1}}
\arrow[yshift=-0.75ex,swap]{rr}{\smatrix{0&0\\ 1&0}} \&\& k^2
\end{tikzcd}\]
which is isomorphic to $R(y^2)$. For $k=\bbF_3$ the corresponding
quiver is the following:
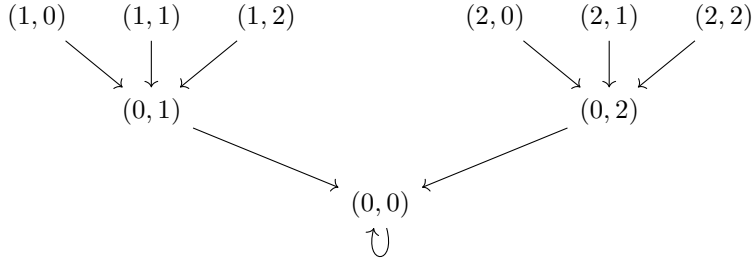
\[\begin{tikzcd}[column sep=small]
(1,0)\ar[dr]&(1,1)\ar[d]&(1,2)\ar[dl]&&(2,0)\ar[dr]&(2,1)\ar[d]&(2,2)\ar[dl]\\
{}&(0,1) \ar[drr]&&&&(0,2) \ar[dll]\\
{}&&&(0,0)\ar[loop below]{r}
\end{tikzcd}\]
\end{exm}

\section{The forgetful functor and its adjoint}

Let $k$ be a finite field with $q$ elements. Composing the forgetful
functor and its left adjoint yields for any quiver $\Ga$ a natural
monomorphism $\Ga \to |k^{(\Ga)}|$.  Combining this with the
description of the quivers arising from indecomposable linear
Kronecker representations, we obtain the following result about finite
quivers.

\begin{cor}
Every finite quiver embeds naturally into a product of quivers whose
connected components belong to the following list:
\begin{enumerate} 
\item linear quivers of type $\A_n$ with $n >0$,
\item cyclic quivers of type $\C_n$ with $n >0$,
\item de Bruijn graphs of dimension $n\ge 0$ on a set with $q$
  elements,
\item almost complete directed trees of width $q$.
\end{enumerate}
\end{cor}
\begin{proof}
  Let $\Ga$ be a finite quiver. Then the linear representation
  $k^{(\Ga)}$ decomposes into a finite direct sum of indecomposable
  representations by the Krull-Remak-Schmidt theorem. It remains to
  observe that direct sums are products in the category of linear
  representations and that the forgetful functor preserves products.
\end{proof}

\section{Linearisation}

Fix a  field $k$. We consider the functors 
\[\Set\lto \Vec_k,\quad X\mapsto k^{(X)},
\qquad\text{and}\qquad
\Set^\op\lto \Vec_k,\quad X\mapsto k^{X}.\]
These functors induce functors between the categories of set valued
and linear Kronecker representations.

 
\begin{lem}\label{le:duality}
For any set $X$ we have a natural isomorphism $k^X\cong \Hom_k(k^{(X)},k)$.
\end{lem}
\begin{proof}
Both functors agree on finite sets and preserve colimits. It remains to
observe that any set is a coproduct of finite sets.
\end{proof}

Let us consider the linear representations corresponding to the linear
quiver $\A_n$.

\begin{prop}
  Let $n\ge 0$. Then we have
  \[k^{(\A_{n+1})}\cong P(n)\qquad \text{and}\qquad k^{\A_{n+1}}\cong
  I(n).\]
\end{prop}
\begin{proof}
  For any $m\ge 0$ we consider the basis $\{x^iy^j\mid i+j=m\}$ of
  $V_m$.  The linear quiver $\A_{n+1}$ identifies with
\[\begin{tikzcd}
\{x^iy^j\mid i+j=n-1\}\arrow[yshift=0.75ex]{rr}{x}
\arrow[yshift=-0.75ex,swap]{rr}{y} && \{x^iy^j\mid i+j=n\}
\end{tikzcd}\]
and this yields the isomorphism $k^{(\A_{n+1})}\xto{\sim} P(n)$. The
second isomorphism follows from the first with Lemma~\ref{le:duality}.
\end{proof}

Next consider the representation corresponding to the cyclic quiver
$\C_n$. It is easily seen that $k^{(\C_n)}\cong k^{\C_n}$ is
isomorphic to 
the regular representation
\[\begin{tikzcd}[ampersand replacement=\&]
k^n\arrow[yshift=0.75ex]{rr}{I_n}
\arrow[yshift=-0.75ex,swap]{rr}{\smatrix{0 & 1 \\
I_{n-1} & 0}} \&\& k^n
\end{tikzcd}\]
which is isomorphic to $R(y^n-1)$. Let $p$ denote the characteristic
of $k$ and write $n=p^cm$ with $m$ coprime to $p$. Then there are
pairwise coprime irreducible polynomials $f_1,\ldots,f_t$ such that
$y^n-1=(y^m-1)^{p^c}=f_1^{p^c}\cdots f_t^{p^c}$. This yields a
decomposition of  $R(y^n-1)$.

\begin{prop}
  Let $n > 0$. Then we have the following decomposition into
  indecomposable regular representations:
\[k^{(\C_n)}\cong R(f_1^{p^c})\oplus\ldots \oplus R(f_t^{p^c})\]
\end{prop}
\begin{proof}
Use that $R(fg)\cong R(f)\oplus R(g)$ when $f$ and $g$ are coprime
homogeneous polynomials.
\end{proof}

\section{Path algebras}

Let  $\Ring$ denote the category of associative rings with unit. Given
a Kronecker representation
\[\begin{tikzcd}
A\arrow[yshift=0.75ex]{r}{f} \arrow[yshift=-0.75ex,swap]{r}{g}& B
\end{tikzcd}\]
in $\Ring$, we view $B$ as a bimodule over $A$ with left
multiplication given by $f$ and right multiplication given by
$g$. Taking this representation to the corresponding tensor
algebra \[T_A(B)=\bigoplus_{n\ge 0}B^{\otimes n}\] yields a functor
\[T\colon \Ring^{\K_2}\lto\Ring.\]

The functor 
\[\Set^\op \times\Ring \lto \Ring,\quad (X,A)\mapsto A^X\]
 induces a functor
\[(\Set^{\K_2})^\op\times\Ring\lto\Ring^{\K_2}.\]

Given a commutative ring $A$ and a quiver $\Ga$, the \emph{path algebra}
$A[\Ga]$ is the free $A$-module with basis consisting of all paths in
$\Ga$ and with multiplication induced by concatenation.

\begin{prop}
  The composition
  \[(\Set^{\K_2})^\op\times\Ring\lto\Ring^{\K_2}\stackrel{T}\lto\Ring\]
  takes a pair $(\Ga,A)$ consisting of a finite quiver and a
  commutative ring to its path algebra $A[\Ga]$.
\end{prop}
\begin{proof}
  Let $\Ga=(\Ga_0,\Ga_1,s,t)$ be a finite quiver. The path algebra
  $A[\Ga]$ identifies with the tensor algebra $T_{A^{\Ga_0}}(B)$ where
  $B$ is the free $A$-module with basis $\Ga_1$.
\end{proof}

\end{document}